\documentclass[reqno,11pt]{amsart}
\usepackage[left=1in,right=1in,top=1in,bottom=1in]{geometry}
\usepackage{enumitem}
\usepackage{amsmath}
\usepackage{amssymb}
\usepackage{mathtools}
\usepackage{graphicx}
\usepackage{tikz}
\usepackage{esint}
\usetikzlibrary{arrows, decorations.markings,matrix,trees}
\usepackage{hyperref}
\hypersetup{
   colorlinks=true,
   citecolor=blue,
   filecolor=blue,
   linkcolor=blue,
   urlcolor=blue
}

\newcommand{\va}{\varepsilon}
\newcommand{\ud}{\mathrm{d}}
\newcommand{\be}{\begin{equation}}
\newcommand{\ee}{\end{equation}}

\newcommand{\B}{\mathbb{B}}

\newcommand{\cL}{\mathcal{L}}

\newcommand{\R}{\mathbb{R}}

\newcommand{\Ss}{\mathbb{S}}

\newcommand{\wt}{\widetilde}

\newcommand{\f}{\frac}

\newcommand{\ol}{\overline}

\DeclareMathOperator{\dist}{dist}

\DeclareMathOperator{\BMO}{BMO}
\DeclareMathOperator{\VMO}{VMO}
\DeclareMathOperator{\sgn}{sgn}
\DeclareMathOperator{\supp}{supp}

\DeclareMathOperator{\Lip}{Lip}

\DeclareMathOperator{\loc}{loc}

\DeclareMathOperator{\id}{id}

\providecommand{\fint}{\mathop{\ooalign{$\int$\cr\hidewidth$-$\hidewidth\cr}}\nolimits}

\def\<{\langle}\def\>{\rangle}
\def\({\left(}\def\){\right)}
\numberwithin{equation}{section}
\theoremstyle{plain}
\newtheorem{thm}{Theorem}[section]

\newtheorem{lem}[thm]{Lemma}
\newtheorem{prop}[thm]{Proposition}

\theoremstyle{definition}
\newtheorem{defn}[thm]{Definition}

\theoremstyle{remark}
\newtheorem{rem}[thm]{Remark}

\title[$\VMO$ approximate fixed points] {A note on an approximate fixed point theorem\\ for $\VMO$ maps}

\date{}

\author{Haotong Fu}
\address{School of Mathematical Sciences, Peking University, Beijing 100871, China}
\email{2301110012@pku.edu.cn}

\author{Wei Wang}
\address{School of Mathematical Sciences, Peking University, Beijing 100871, China}
\email{wwmath166@outlook.com,\,\,2201110024@stu.pku.edu.cn}

\begin{document}

\begin{abstract}
We prove an approximate fixed point theorem for maps of vanishing
mean oscillation from the unit ball to itself. This gives a short proof of the Open Problem 5.10 of Brezis \cite{Bre23}.
\end{abstract}

\subjclass[2020]{Primary 46E30; Secondary 42B35, 47H10.}

\keywords{$\VMO$, $\BMO$, approximate fixed point, Brouwer's theorem, no-retraction theorem}
\maketitle

\section{Introduction}

\subsection{Main results}

Let $n\geq 1$ and let
\[
\B=B_1^n=\{x\in\R^n: |x|<1\}.
\]
For $u\in L_{\loc}^1(\B;\R^m)$ and for a ball
$B_r(a)\subset\subset\B$, set
\[
u_{B_r(a)}:=\fint_{B_r(a)}u\ud x.
\]
The space $\BMO$ was introduced by John and Nirenberg \cite{JN61} and
$\VMO$ by Sarason \cite{Sar75}. In the present paper, the $\BMO$ seminorm
on $\B$ is
\[
\|u\|_{\BMO(\B)}:=\sup_{B_r(a)\subset\subset\B}
\fint_{B_r(a)}|u-u_{B_r(a)}|\ud x,
\]
and the corresponding small-scale oscillation is
\[
\omega_u(\rho):=
\sup_{\substack{B_r(a)\subset\subset\B\\0<r<\rho}}
\fint_{B_r(a)}|u-u_{B_r(a)}|\ud x.
\]

\begin{defn}
For $u\in L_{\loc}^1(\B;\R^m)$, we define:
\begin{enumerate}
\item $u\in\BMO(\B;\R^m)$ if $\|u\|_{\BMO(\B)}<+\infty$;
\item $u\in\VMO(\B;\R^m)$ if $u\in\BMO(\B;\R^m)$ and
$\omega_u(\rho)\to0^+$ as $\rho\downarrow0$. 
\end{enumerate}
\end{defn}

We now present the following main theorem.

\begin{thm}\label{thm:main}
Let $f\in\VMO(\B;\R^n)\cap L^\infty(\B;\R^n)$ satisfy
\[
|f(x)|\leq1\quad\text{for a.e. }x\in\B.
\]
Then, for any $\va>0$,
\[
\cL^n(\{x\in\B: |f(x)-x|<\va\})>0,
\]
where $ \cL^n $ denotes the Lebesgue measure on $ \R^n $.

Equivalently, there do not exist $\va_0>0$ and such a map $f$ satisfying
\[
        |f(x)-x|\geq\va_0
        \quad\text{for a.e. }x\in\B.
\]
\end{thm}

The above theorem can be viewed as the $\VMO$ counterpart of
Brouwer's fixed point theorem. Recall the classical
statement as follows.

\begin{thm}\label{classical}
If $f\in C^0(\ol{\B},\ol{\B})$, then there exists
$x_0\in\ol{\B}$ such that $f(x_0)=x_0$.
\end{thm}

Indeed, if $f(x)\neq x$ for all $x\in\ol{\B}$, then following the ray from
$f(x)$ through $x$ gives a continuous retraction of $\ol{\B}$ onto
$\Ss^{n-1}$. The impossibility of such a retraction proves Brouwer's
theorem; see, for instance, \cite[Corollary 2.15]{Hat02}.

\begin{rem}
We make the following comments on Theorem \ref{thm:main}.
\begin{enumerate}
\item Brezis asked in Open Problem 5.10 of \cite{Bre23} whether an
approximate fixed point statement remains true for $\VMO$ self-maps in closed unit balls. We deal with the problem on the open unit ball $ \B $, since the $ \VMO $ maps are equivalent up to a set of zero measures. 
\item Lanconelli and Uguzzoni proved a statement equivalent to Theorem \ref{thm:main}; see Proposition 12 of \cite{LU02}. Their proof is based on
the $\VMO$ degree theory developed in that paper. Here we give a direct
proof in the Euclidean ball, using only local averages, a ray-exit
construction, and the classical no-retraction form of Brouwer's theorem.

\item The proof below is self-contained apart from the classical Brouwer's
fixed point theorem, and it does not use $\VMO$ degree theory.
\end{enumerate}
\end{rem}

\section{The ray-exit construction}

We first record the elementary geometric construction that replaces the
pointwise framework. Throughout this section, fix
$0<\eta<2$ and set
\[
K_\eta=\{(x,y)\in\ol{\B}\times\ol{\B}:|x-y|\geq\eta\}.
\]
For $(x,y)\in K_\eta$, the ray
\[
y+t(x-y),\quad t\geq0,
\]
meets the sphere at a unique last point. We denote this point by $P(x,y)$.
The following lemma gives the explicit formula and the regularity property
of $ P $.

\begin{lem}\label{lem:ray-map}
For $(x,y)\in K_\eta$, put $v=x-y$. Then
\[
        P(x,y)=y+t(x,y)v,
\]
where
\[
        t(x,y):=
        \f{-y\cdot v+
        \sqrt{(y\cdot v)^2+|v|^2(1-|y|^2)}}{|v|^2}.
\]
Moreover:
\begin{enumerate}[label=$(\theenumi)$]
\item $P(x,y)\in\Ss^{n-1}$ and $t(x,y)\geq1$;
\item if $x\in\Ss^{n-1}$, then $P(x,y)=x$;
\item $P:K_\eta\to\Ss^{n-1}$ is Lipschitz.
\end{enumerate}
\end{lem}

\begin{proof}
The intersections of the line $y+tv$ with $\Ss^{n-1}$ are the roots of
\[
        |y+tv|^2=1,
\]
or equivalently
\[
        |v|^2t^2+2(y\cdot v)t+|y|^2-1=0.
\]
The displayed value of $t(x,y)$ is the larger root and, therefore, gives the
last point at which the ray meets the closed unit ball. Hence
$P(x,y)\in\Ss^{n-1}$.

Since $x,y\in\ol{\B}$ and $\ol{\B}$ is convex, the segment
$y+t(x-y)$, $0\leq t\leq1$, lies in $\ol{\B}$. Therefore, the exit parameter
is at least $1$. If $x\in\Ss^{n-1}$, then $t=1$ is already the exit
parameter, so the larger root is $1$ and $P(x,y)=x$.

It remains to prove the Lipschitz assertion. On $K_\eta$ one has
$|v|\geq\eta$, so the denominator in the formula for $t$ is bounded away
from zero. The discriminant
\[
        \Delta=(y\cdot v)^2+|v|^2(1-|y|^2)
\]
is strictly positive on $K_\eta$. Indeed, if $\Delta=0$, then
$|y|=1$ and $y\cdot v=0$, whence
\[
        |x|^2=|y+v|^2
        =|y|^2+2y\cdot v+|v|^2
        =1+|v|^2>1,
\]
contrary to $x\in\ol{\B}$. Since $K_\eta$ is compact, there exists
$c_\eta>0$ such that $\Delta\geq c_\eta$ on $K_\eta$. The formula defining
$P$ is therefore $C^1$ on a neighborhood of $K_\eta$, with a bounded
derivative. The mean-value theorem gives the Lipschitz bound.
\end{proof}

The next estimate is the geometric input that later yields the identity
trace on the boundary.

\begin{lem}\label{lem:boundary-geometry}
For any $(x,y)\in K_\eta$,
\[
        |P(x,y)-x|
        \leq \frac{2}{\eta}(1-|x|).
\]

In particular, if $f:\B\to\R^n$ is a measurable map such that
$(x,f(x))\in K_\eta$ for a.e. $x\in\B$ and
$r(x)=P(x,f(x))$, then, for a.e. $x\in \B\backslash\{0\}$,
\[
\left|r(x)-\frac{x}{|x|}\right|
\leq \left(1+\frac{2}{\eta}\right)(1-|x|).
\]
\end{lem}

\begin{proof}
Fix $(x,y)\in K_\eta$ and write
\[
        d=|x-y|,\quad e=\frac{x-y}{|x-y|}.
\]
Since $P(x,y)$ lies on the ray from $y$ through $x$ and beyond $x$, there is
$a\geq0$ such that
\[
        P(x,y)=x+ae.
\]
Let the other intersection point of the same line with $\Ss^{n-1}$ be
$x-be$, with $b\geq0$. The point $y=x-de$ belongs to $\ol{\B}$, and
$d\geq\eta$. Since the intersection of this line with $\ol{\B}$ is exactly
the segment between its two endpoints on the sphere, we have $b\geq d$.

The two roots $s=a$ and $s=-b$ of
\[
        |x+se|^2=1
\]
satisfy
\[
        (s-a)(s+b)=s^2+2(x\cdot e)s+|x|^2-1.
\]
Comparing constant terms gives $ab=1-|x|^2$. Hence
\[
        a=\frac{1-|x|^2}{b}
        \leq \frac{1-|x|^2}{\eta}
        \leq \frac{2}{\eta}(1-|x|),
\]
which proves the first estimate.

If $r(x)=P(x,f(x))$, then
\[
\left|r(x)-\frac{x}{|x|}\right|
\leq |r(x)-x|+\left|x-\frac{x}{|x|}\right|
\leq \frac{2}{\eta}(1-|x|)+(1-|x|),
\]
for any $x\neq0$.
\end{proof}

\section{The \texorpdfstring{$\VMO$}{} sphere-valued map associated with a separated counterexample}

We now assume, in contradiction, that there exist $\eta>0$ and
$f\in\VMO(\B;\R^n)\cap L^\infty(\B;\R^n)$ such that
\be
|f(x)|\leq1,\quad |f(x)-x|\geq\eta
\quad\text{for a.e. }x\in\B.
\label{eq:separation}
\ee
The case $\eta\geq2$ is impossible because
\[
|f(x)-x|\leq |f(x)|+|x|<2\quad\text{ for a.e. }x\in\B.
\]
Thus, after reducing
$\eta$ if necessary, we may assume $0<\eta<2$.

\begin{lem}\label{lem:vmo-retraction}
Under the assumption \eqref{eq:separation}, the map
\[
r(x)=P(x,f(x))
\]
is well defined a.e. in $ \B $ and belongs to $\VMO(\B;\Ss^{n-1})$. 

Moreover,
\be\label{eq:r-boundary-estimate}
\left|r(x)-\f{x}{|x|}\right|
\leq A_\eta(1-|x|)
\quad\text{for a.e. }x\in\B\backslash\{0\},
\ee
where
\[
A_\eta=1+\f{2}{\eta}.
\]
\end{lem}

\begin{proof}
By Lemma \ref{lem:ray-map}, $P$ is Lipschitz on $K_\eta$. Extending each
component of $P$ from $K_\eta$ to $\R^{2n}$ by the McShane extension
theorem (see \cite[Theorem 5.1]{Sim83}), we obtain a global Lipschitz map
$\widetilde P:\R^{2n}\to\R^n$ such that $\widetilde P=P$ on $K_\eta$ and $ \Lip\wt{P}=\Lip P $.

The map $x\mapsto (x,f(x))$ belongs to $\VMO(\B;\R^{2n})$. Indeed, for any
ball $B_\rho(a)\subset\subset\B$,
\[
        \fint_{B_\rho(a)} |x-x_{B_\rho(a)}|\ud x
        \leq \rho.
\]
Thus, its small-scale mean oscillation is bounded by $\omega_f(\rho)+\rho$.
The composition with a global Lipschitz map is preserved $\VMO$. Since
$(x,f(x))\in K_\eta$ for a.e. $x\in\B$, we have
\[
        \widetilde P(x,f(x))=P(x,f(x))\in\Ss^{n-1}
        \quad\text{for a.e. }x.
\]
Thus $r\in\VMO(\B;\Ss^{n-1})$. The estimate \eqref{eq:r-boundary-estimate} follows from Lemma \ref{lem:boundary-geometry}.
\end{proof}

\section{Mollifying sphere-valued \texorpdfstring{$\VMO$}{} maps}

The next step is to replace the a.e. defined sphere-valued map by continuous
sphere-valued maps on compact subsets. The key observation here is that local averages of
a $\VMO$ map with values in the sphere do not approach the origin at small
scales.

Let $\rho\in C_0^\infty(\B)$ be non-negative and satisfy
\[
\int_{\R^n}\rho\ud x=1.
\]
Set $\rho_\delta(x)=\f{1}{\delta^{n}}\rho(\f{x}{\delta})$.

\begin{lem}\label{lem:mollification-nonzero}
Let $u\in\VMO(\B;\Ss^{n-1})$. For any $\alpha>0$ there exists
$\delta_0>0$ such that, whenever $0<\delta<\delta_0$ and
$B_\delta(x)\subset\subset\B$, the local average
\[
u_\delta(x)=\int_{\B}\rho_\delta(x-y)u(y)\ud y
\]
satisfies
\[
|u_\delta(x)|\geq1-\alpha.
\]

In particularly, if $K\subset\B$ is compact and
$0<\delta<\min\{\delta_0,\dist(K,\partial\B)\}$, then the normalized
mollification
\[
R_\delta(x)=\f{u_\delta(x)}{|u_\delta(x)|}
\]
is well defined and continuous on $K$.
\end{lem}

\begin{proof}
Fix $x$ with $B_\delta(x)\subset\subset\B$ and write $B_x=B_\delta(x)$.
Since $\rho_\delta\leq C\delta^{-n}$ and
$\supp\rho_\delta(x-\cdot)\subset B_x$,
\[
        \int \rho_\delta(x-y)|u(y)-u_{B_x}|\ud y
        \leq
        C\fint_{B_x}|u-u_{B_x}|\ud y
        \leq C\omega_u(\delta).
\]
Consequently,
\[
        |u_\delta(x)-u_{B_x}|
        \leq C\omega_u(\delta),
\]
and hence
\[
        \int\rho_\delta(x-y)|u(y)-u_\delta(x)|\ud y
        \leq C\omega_u(\delta).
\]
Since $|u(y)|=1$ for a.e. $y\in\B$,
\[
        1-|u_\delta(x)|
        \leq
        \int\rho_\delta(x-y)\big||u(y)|-|u_\delta(x)|\big|\ud y
        \leq
        \int\rho_\delta(x-y)|u(y)-u_\delta(x)|\ud y
        \leq C\omega_u(\delta).
\]
Because $u\in\VMO$, $\omega_u(\delta)\to0^+$ as $\delta\downarrow0$. This
proves the non-vanishing property. The continuity of $u_\delta$ follows from
the smoothness of the mollifier and the bound $u\in L^\infty$.
\end{proof}

We now apply the preceding lemma to the particular map obtained from the
separated counterexample. The boundary estimate from Lemma
\ref{lem:vmo-retraction} ensures that the normalized mollification agrees,
up to a small error, with the radial projection near the boundary.

\begin{lem}\label{lem:mollified-boundary}
Let $r$ be the map obtained in Lemma \ref{lem:vmo-retraction}. There exist
$\tau\in(0,\f{1}{10})$ and $\delta\in(0,\tau)$ such that the normalized
mollification
\[
R_\delta(x)=\f{r_\delta(x)}{|r_\delta(x)|},\quad
r_\delta(x)=\int_{\B}\rho_\delta(x-y)r(y)\ud y,
\]
is well defined on $\ol B_{1-2\tau}$ and satisfies
\be\label{eq:R-close-to-radial}
\left|R_\delta(x)-\f{x}{|x|}\right|<\f12
\ee
for any $x$ with $|x|\in[1-3\tau,1-2\tau]$.
\end{lem}

\begin{proof}
Choose $\tau\in(0,\f{1}{10})$ so small that
$4A_\eta\tau<\f{1}{32}$, where $A_\eta$ is as in Lemma
\ref{lem:vmo-retraction}. Then choose $\delta\in(0,\f{\tau}{4})$ so small
that Lemma \ref{lem:mollification-nonzero} gives
$|r_\delta|\geq\f{3}{4}$ on $\ol B_{1-2\tau}$.

Let $x$ satisfy $|x|\in[1-3\tau,1-2\tau]$. If $|x-y|<\delta$, then
$|y|\geq1-4\tau>\f{1}{2}$ and
\[
1-|y|\leq 1-|x|+|x-y|\leq 3\tau+\delta.
\]
Using \eqref{eq:r-boundary-estimate}, we get for a.e. such $ y $
\[
\left|r(y)-\f{y}{|y|}\right|
\leq A_\eta(3\tau+\delta).
\]
Moreover, since $|x|,|y|\geq\f{1}{2}$,
\[
\left|\f{y}{|y|}-\f{x}{|x|}\right|\leq C|x-y|\leq C\delta.
\]
Therefore
\[
\left|r_\delta(x)-\f{x}{|x|}\right|
\leq A_\eta(3\tau+\delta)+C\delta.
\]
After decreasing $\delta$, the right-hand side is less than $\f{1}{8}$.
Since $|r_\delta(x)|\geq\f{3}{4}$ and $\f{x}{|x|}$ has norm $1$, we also have
\[
\left|R_\delta(x)-r_\delta(x)\right|
=\left|1-|r_\delta(x)|\right|
\leq \left|r_\delta(x)-\f{x}{|x|}\right|.
\]
Therefore
\[
\left|R_\delta(x)-\f{x}{|x|}\right|
\leq2\left|r_\delta(x)-\f{x}{|x|}\right|
<\f{1}{2}.
\]
This proves \eqref{eq:R-close-to-radial}.
\end{proof}

\section{Construction of a continuous retraction}

We now use the normalized mollification to construct a genuine continuous
retraction of the closed ball onto the sphere. This is the step in which
mollification replaces the use of degree theory for $\VMO$ maps.

\begin{prop}\label{prop:continuous-retraction}
The assumption \eqref{eq:separation} implies the existence of a continuous map
\[
        \mathcal R:\ol\B\to\Ss^{n-1}
\]
such that
\[
        \mathcal R(x)=x
        \quad\text{for any }x\in\partial\B.
\]
\end{prop}

\begin{proof}
Let $\tau$ and $\delta$ be chosen as in Lemma \ref{lem:mollified-boundary},
and let $R_\delta=\f{r_\delta}{|r_\delta|}$ on $\ol B_{1-2\tau}$. Choose a
continuous cutoff $\chi:[0,1]\to[0,1]$ such that
\[
\chi(s)=0\quad\text{for }s\leq1-3\tau,\quad
\chi(s)=1\quad\text{for }s\geq1-2\tau.
\]
For $x\neq0$, put
\[
\nu(x)=\f{x}{|x|}.
\]
On the annulus $|x|\in[1-3\tau,1-2\tau]$, define
\[
W(x)=(1-\chi(|x|))R_\delta(x)+\chi(|x|)\nu(x).
\]
By \eqref{eq:R-close-to-radial}, the vectors $R_\delta(x)$ and $\nu(x)$ are
in a common open hemisphere. Equivalently,
\[
        R_\delta(x)\cdot\nu(x)>0,
\]
and therefore $W(x)\neq0$ throughout the interpolation annulus.

Define
\[
        \mathcal R(x)=
        \begin{cases}
        R_\delta(x),& |x|\in[0,1-3\tau),\\[0.4em]
        \dfrac{W(x)}{|W(x)|},&|x|\in[1-3\tau,1-2\tau],\\[0.9em]
        \dfrac{x}{|x|},&|x|\in(1-2\tau,1].
        \end{cases}
\]
The three definitions agree on the two interfaces by the choice of $\chi$.
Thus, $\mathcal R$ is continuous on $\ol\B$, takes values in $\Ss^{n-1}$,
and satisfies $\mathcal R(x)=x$ on $\partial\B$.
\end{proof}

The contradiction will be obtained from the following standard no-retraction
consequence of Brouwer's fixed point theorem.

\begin{lem}\label{lem:no-cont-retraction}
There is no continuous map
\[
        \mathcal R:\ol\B\to\Ss^{n-1}
\]
such that $\mathcal R(x)=x$ for any $x\in\partial\B$.
\end{lem}

\begin{proof}
If such a map existed, define
\[
        T(x)=-\mathcal R(x),
        \quad x\in\ol\B.
\]
Then $T$ is a continuous map from $\ol\B$ to $\ol\B$. By the classical
Brouwer's fixed point theorem, Theorem \ref{classical}, there exists
$x_0\in\ol\B$ such that $T(x_0)=x_0$. Since $T(x_0)\in\Ss^{n-1}$, one has
$x_0\in\Ss^{n-1}$. On the boundary, $\mathcal R(x_0)=x_0$, and therefore
\[
x_0=T(x_0)=-\mathcal R(x_0)=-x_0,
\]
which is impossible because $|x_0|=1$.
\end{proof}

\section{Proof of the main theorem}

\subsection{Proof of Theorem \ref{thm:main}}

Assume by contradiction that for some $\eta>0$,
\[
        |f(x)-x|\geq\eta
        \quad\text{for a.e. }x\in\B.
\]
By Lemma \ref{lem:vmo-retraction}, this produces a map
$r\in\VMO(\B;\Ss^{n-1})$ that satisfies the quantitative boundary estimate
\eqref{eq:r-boundary-estimate}. Lemmas \ref{lem:mollification-nonzero} and
\ref{lem:mollified-boundary}, followed by Proposition
\ref{prop:continuous-retraction}, then produce a continuous retraction
\[
\mathcal R:\ol\B\to\Ss^{n-1},\quad
\mathcal R|_{\partial\B}=\id_{\partial\B}.
\]
This contradicts the classical no-retraction Lemma
\ref{lem:no-cont-retraction}. Hence, such $\eta$ does not exist.

If, for some $\va>0$, the set
\[
        \{x\in\B: |f(x)-x|<\va\}
\]
had measure zero, then $|f(x)-x|\geq\va$ for a.e. $x\in\B$,
contradicting what we have just proved. Therefore,
\[
\cL^n(\{x\in\B: |f(x)-x|<\va\})>0
\]
for any $\va>0$.

\subsection{The one-dimensional case}

When $n=1$, we give a direct proof without using Theorem \ref{classical}.

\begin{lem}\label{lem:one-dimensional}
Every map $u\in\VMO([-1,1];\{-1,1\})$ is almost everywhere constant on $(-1,1)$.
Consequently, there is no $r\in\VMO([-1,1];\{-1,1\})$ satisfying
\[
        \lim_{\delta\downarrow0}
        \frac1{2\delta}
        \int_{(-1,-1+\delta)\cup(1-\delta,1)}
        |r(x)-\sgn x|\ud x
        =0.
\]
\end{lem}

\begin{proof}
Let $u:(-1,1)\to\{-1,1\}$ be measurable and suppose that $u$ is not a.e.
constant. Set
\[
E=\{x:u(x)=1\}.
\]
Then both $E$ and its complement have positive measures. Choose a Lebesgue
density point $a\in E$ and a Lebesgue density point
$b\in (-1,1)\backslash E$, with $a\neq b$.

For $h>0$ small enough so that the intervals below are contained in
$(-1,1)$, define
\[
\theta_h(c)=\frac{\cL^1(E\cap(c-h,c+h))}{2h}.
\]
The function $\theta_h$ is continuous in $c$, since it is the convolution of
$1_E$ with the normalized interval indicator $(2h)^{-1}1_{(-h,h)}$. By the
density properties of $a$ and $b$, for all sufficiently small $h$,
\[
\theta_h(a)>\frac{3}{4},\quad\theta_h(b)<\frac{1}{4}.
\]
The intermediate value theorem gives a point $c_h$ between $a$ and $b$ such
that $\theta_h(c_h)=\f{1}{2}$. Put $J_h=(c_h-h,c_h+h)$. Then exactly half of
$J_h$, up to a null set, lies in $E$, and hence $u_{J_h}=0$. Consequently,
\[
        \fint_{J_h}|u-u_{J_h}|\ud x
        =
        \fint_{J_h}|u|\ud x
        =1.
\]
Letting $h\downarrow0$ contradicts the $\VMO$ condition. Thus, any
$\{-1,1\}$-valued $\VMO$ map on $(-1,1)$ is a.e. constant.

If $r\equiv1$ a.e., then on $(-1,-1+\delta)$ the integrand
$|r-\sgn x|$ equals $2$, while on $(1-\delta,1)$ it equals $0$. The
normalized two-sided average is, therefore, $1$. The case $r\equiv-1$ is the
same, with the two endpoints interchanged. Hence, the limit displayed cannot 
be zero.
\end{proof}

\begin{proof}[Proof of Theorem \ref{thm:main} in one-dimensional case]
If $n=1$, the above separated-counterexample construction gives
\[
r\in\VMO([-1,1];\{-1,1\}).
\]
Lemma \ref{lem:boundary-geometry} gives
\be
\left|r(x)-\frac{x}{|x|}\right|
\leq
\left(1+\frac2\eta\right)(1-|x|)
\quad\text{for a.e. }x\in (-1,1)\backslash\{0\}.
\label{eq:boundary-final}
\ee
Since $\f{x}{|x|}=\sgn x$ in one dimension, \eqref{eq:boundary-final}
implies
\[
\frac1{2\delta}\int_{(-1,-1+\delta)\cup(1-\delta,1)}
|r(x)-\sgn x|\ud x
\leq
\left(1+\frac2\eta\right)\delta,
\]
which tends to zero. This contradicts Lemma \ref{lem:one-dimensional}.
\end{proof}

\section*{Acknowledgment}
The authors acknowledge the use of ChatGPT, developed by OpenAI, for
assistance with checking intermediate estimates and improving the
exposition. All AI-assisted suggestions were independently verified by the
authors, who take full responsibility for the mathematical content.

\bibliographystyle{plain}

\end{document}